\documentclass[11pt,final,reqno]{amsart}
\usepackage[english]{babel}
\usepackage{amsmath,amsthm}
\usepackage{amssymb}
\usepackage{amsfonts}
\usepackage{graphics}
\usepackage{inputenc}

\usepackage{mathrsfs}
\usepackage{mathrsfs}
\usepackage{color}

\usepackage{float}
\usepackage[left=2.0cm,right=2.0cm,top=3.2cm,bottom=3.3cm]{geometry}
\usepackage{mathrsfs}
\usepackage{color}
\usepackage{graphicx}%
\usepackage{color}
\usepackage{cite}
\usepackage{float}
\usepackage{hyperref}
\hypersetup{
   colorlinks = true,%
   citecolor = blue,
   filecolor=red,%
   linkcolor = [rgb]{0.65,0.0,0.0},%
   anchorcolor = red,
   pagecolor = red,
   urlcolor= [rgb]{0.65,0.0,0.0}
   linktocpage=true,
   pdfpagelabels=true,
   bookmarksnumbered=true,
}

\usepackage{flafter}
\usepackage{mathrsfs}

\newtheorem{theorem}{Theorem}[section]
\newtheorem{lemma}{Lemma}[section]
\newtheorem{corollary}{Corollary}[section]

\newtheorem{remark}{Remark}[section]

\numberwithin{equation}{section}

\newcommand{\eps}{{\varepsilon}}
\newcommand{\ds}{\displaystyle}

\numberwithin{equation}{section}

\numberwithin{equation}{section}

\newcommand{\R}{{\mathbb R}}
\newcommand{\Sg}{\mathbb{S}_+^{n}}
\newcommand{\G}{G^n_{j,\tau_j}}

\newcommand{\im}[1]{\displaystyle\int_{M}{#1}\,{\rm d}v_g}
\newcommand{\ig}[1]{\displaystyle\int_{\Sg}{#1}\,{\rm d}v_g}
\newcommand{\igo}[1]{\displaystyle\int_{\Sg}{#1}\,{\rm d}v_{g_0}}
\makeatletter
\renewcommand*{\@fnsymbol}[1]{\ensuremath{\ifcase#1\or *\or **\or \ddagger\or
   \mathsection\or \mathparagraph\or \|\or **\or \dagger\dagger
   \or \ddagger\ddagger \else\@ctrerr\fi}}

\g@addto@macro{\endabstract}{\@setabstract}
\newcommand{\authorfootnotes}{\renewcommand\thefootnote{\@fnsymbol\c@footnote}}%
\makeatother
\def\d{{{\rm d}%
}v_{g}}

\makeatletter
\renewcommand{\@biblabel}[1]{[#1]\hfill}
\makeatother

\title[Multipolar Hardy inequality on Riemannian manifolds]{Multipolar Hardy inequalities on Riemannian manifolds}
\author{Francesca Faraci}
\address{Department of Mathematics and Informatics, University of Catania}
\email{ffaraci@dmi.unict.it}
\author{Csaba Farkas}
\address{Department of Mathematics and Informatics, Sapientia University, Tg. Mure\c s, Romania, \newline Institute of Applied Mathematics,
	\'Obuda University, 1034 Budapest, Hungary}
\email{farkas.csaba2008@gmail.com}
\author{Alexandru Krist\'aly}
\address{\noindent Department of Economics, Babe\c s-Bolyai University, Cluj-Napoca, Romania, \newline \noindent Institute of Applied Mathematics,
	\'Obuda University, 1034 Budapest, Hungary}
\email{alexandrukristaly@yahoo.com; alex.kristaly@econ.ubbcluj.ro}
\begin{document}
\maketitle

\begin{center}
	{\it Dedicated to 
		Professor Enrique Zuazua on the occasion of his 55th birthday}
\end{center}

\begin{abstract}
We prove multipolar Hardy inequalities on  complete Riemannian mani\-folds, providing various curved counterparts of some Euclidean multipolar inequalities due to Cazacu and Zuazua [Improved multipolar Hardy inequalities, 2013]. We notice that our inequalities deeply depend on the curvature, providing (quantitative) information about the deflection from the flat case. 
By using these inequalities together with variational methods and group-theoretical arguments, we also establish non-existence,  existence and multiplicity results for certain Schr\"odinger-type  problems involving the Laplace-Beltrami operator and bipolar potentials on Cartan-Hada\-mard manifolds and on the open upper hemisphere, respectively. 
\end{abstract}

\maketitle

\section{Introduction}

The classical {\it unipolar Hardy inequality} (or, uncertainty principle) states that if  $n\ge 3$, then 
$$
\int_{\mathbb R^n}|\nabla u|^2{\rm d}x \geq \frac{(n-2)^2}{4}\int_{\mathbb
	R^n}\frac{u^2}{|x|^2}{\rm d}x,\ \forall u\in C_0^\infty(\mathbb R^n);
$$
 here, the constant $\frac{(n-2)^2}{4}$ is sharp and not achieved. 
 Many efforts have been made over the last two decades 
to improve/extend Hardy inequalities in various directions. 
One of the most challenging research topics in this direction  is the so-called  {\it multipolar Hardy inequality}. 
Such kind of extension is motivated by molecular physics and quantum chemistry/cosmology. Indeed, by describing the  behavior of electrons and atomic nuclei in a molecule within the theory of Born-Oppenheimer approximation or Thomas-Fermi theory,  particles can be  modeled as certain singularities/poles $x_1,...,x_m\in \mathbb R^n$,   producing their effect within the form $x\mapsto |x-x_i|^{-1}$, $i\in \{1,...,m\}$. Having such mathematical models, several authors studied the behavior of the operator with inverse square potentials with multiple poles, namely $$\ds \mathscr{L}:=-\Delta-\sum_{i=1}^m\frac{\mu_i^+}{|x-x_i|^2},$$
see
Bosi, Dolbeaut and Esteban \cite{Bosi-Dolbeaut-Esteban-CPAA-2008},  
Cao and Han \cite{Cao-Han}, Felli, Marchini and  Terracini \cite{Felli-Marchini-Terracini-2007JFA}, Guo, Han and Niu \cite{Kinaiak},  Lieb \cite{Lieb}, 
Adimurthi  \cite{Adimurthi-2013-CCM}, and references therein.
Very recently, Cazacu and Zuazua \cite{cazacuzuazua} proved  an optimal multipolar counterpart of the above (unipolar) Hardy inequality, i.e., 
\begin{equation}\label{MHIEUC}\int_{\mathbb{R}^n}|\nabla u|^2{\rm d}x\geq\frac{(n-2)^2}{m^2}\sum_{1\leq i<j\leq m}\int_{\mathbb{R}^n}\frac{|x_i-x_j|^2}{|x-x_i|^2|x-x_j|^2}u^2{\rm d}x,\ \ \forall u\in C_0^\infty(\mathbb R^n),
\end{equation}
where $n\geq 3$, and $x_1,...,x_m\in \mathbb R^n$ are different poles; moreover, the constant $\frac{(n-2)^2}{m^2}$ is optimal. By using the paralelogrammoid law, (\ref{MHIEUC}) turns to be equivalent to
\begin{equation}\label{paramultipolar}
\int_{\mathbb{R}^n}|\nabla u|^2{\rm d}x\geq\frac{(n-2)^2}{m^2}\sum_{1\leq i<j\leq m}\int_{\mathbb{R}^n}\left|\frac{x-x_i}{|x-x_i|^2}-\frac{x-x_j}{|x-x_j|^2}\right|^2u^2{\rm d}x,\ \ \forall u\in C_0^{\infty}(\mathbb{R}^n).
\end{equation}  

All of the aforementioned works considered the flat/isotropic setting where no external force is present. Once the ambient space structure is perturbed, coming for instance by a magnetic or gravitational field, the above results do not provide a full description of the physical phenomenon  due to the presence of the {\it curvature}. 

In order to discuss such a curved setting, we put ourselves into the Riemannian realm, i.e., we consider an $n(\geq 3)$-dimensional complete
Riemannian manifold  $(M, g)$,   $d_g : M \times M \to [0, \infty)$ is its usual distance function associated to the Riemannian metric $g$, 
${\rm d}v_g$ is its canonical volume element, $\exp_x:T_xM\to M$ is its standard exponential map, 
and  $\nabla_g u(x)$ is the gradient of a function $u:M\to \mathbb R$ at $x\in M$, respectively. 
Clearly, in the curved setting of $(M,g)$, the vector $x-x_i$ and distance $|x-x_i|$ should be reformulated into a  geometric context by considering $\exp_{x_i}^{-1}(x)$ and $d_g(x,x_i),$ respectively. Note that  $$\nabla_g d_g(\cdot,y)(x)=-\frac{\exp_x^{-1}(y)}{d_g(x,y)}\text{\ for every }y\in M, \text{ }x\in M\setminus (\{y\}\cup {\rm cut}(y)),$$
where cut$(y)$ denotes the cut-locus of $y$ on $(M,g).$ In this setting, a natural question arises: if $\Omega\subseteq M$ is an open domain and $S=\{x_1,...,x_m\}\subset \Omega$ is the set of distinct poles, can we prove
\begin{equation}\label{natural-MHIEUC}
\int_{\Omega}|\nabla_g u|^2{\rm d}v_g\geq\frac{(n-2)^2}{m^2}\sum_{1\leq i<j\leq m}\int_{\Omega}V_{ij}(x)u^2{\rm d}x,\ \ \forall u\in C_0^\infty(\Omega),
\end{equation}
where $$V_{ij}(x)=\frac{d_g(x_i,x_j)^2}{d_g(x,x_i)^2d_g(x,x_j)^2}\ \ {\rm or}\ \  V_{ij}(x)=\left|\frac{\nabla_g d_g(x,x_i)}{d_g(x,x_i)}-\frac{\nabla_g d_g(x,x_j)}{d_g(x,x_j)}\right|^2?$$
Clearly, in the Euclidean space $\mathbb R^n$, inequality (\ref{natural-MHIEUC}) corresponds to  (\ref{MHIEUC})
and (\ref{paramultipolar}),  for the above choices of $V_{ij}$, respectively. It turns out that the answer deeply depends on the curvature of the Riemannian manifold $(M,g)$. Indeed, if the Ricci curvature verifies Ric$(M,g)\geq c_0(n-1)g$ for some $c_0>0$ (as in the case of the $n$-dimensional unit sphere $\mathbb S^n$), we know by the theorem of Bonnet-Myers that $(M,g)$ is compact; thus, we may use the constant functions $u\equiv c\in \mathbb R$ as test-functions in (\ref{natural-MHIEUC}), and we get a contradiction.  However, when $(M,g)$ is a Cartan-Hadamard manifold (i.e., complete, simply connected Riemannian manifold with non-positive sectional curvature), we can expect the validity of (\ref{natural-MHIEUC}), see Theorems \ref{main-main1} \& \ref{masodik} and suitable Laplace comparison theorems, respectively.

Accordingly, the primary aim of the present paper is to investigate {\it multipolar Hardy inequalities on complete Riemannian manifolds}. We emphasize that such a study requires new  technical and theoretical approaches. In fact, we need to explore those geometric and analytic properties which are behind of the theory of  multipolar Hardy inequalities in the flat context, formulated now  in terms of curvature, geodesics, exponential map, etc. 
We notice that striking results were also achieved  recently in the theory of unipolar Hardy-type  inequalities on {curved spaces}. The pioneering work of Carron \cite{CarronHardy}, who  studied Hardy inequalities on complete  non-compact Riemannian manifolds, opened new perspectives in the study of functional inequalities with singular terms on curved spaces. Further contributions have been provided by D'Ambrosio and Dipierro \cite{dambrosiohardyRiemann}, Kombe and \"Ozaydin \cite{Kombe-Ozaydin-TAMS2019, Kombe-Ozaydin-TAMS2013}, Xia \cite{XIA-JMAA_2014}, and Yang, Su and Kong \cite{Yang-Su-Kong-CCM2014}, where various improvements of the usual Hardy inequality is presented on complete, non-compact Riemannian manifolds. Moreover, certain unipolar Hardy and Rellich type inequalities were obtained on non-reversible Finsler manifolds by Farkas, Krist\'aly and Varga  \cite{FKVCalculus}, and Krist\'aly and Repovs \cite{Kristaly-Repovs-2016-CCM}.

 In the sequel we shall present our results; for further use, 
let 
$\Delta_g$ be the Laplace-Beltrami operator on $(M,g)$. 
 Let $m\geq 2$,  $S=\{x_1,...,x_m\}\subset M$ be the set of poles with $x_i\neq x_j$ if $i\neq j,$ and for simplicity of notation,  let $d_i=d_g(\cdot, x_i)$ for every $i\in \{1,...,m\}.$ Our first result reads as follows.

\begin{theorem}[\textbf{Multipolar Hardy inequality I}]\label{main-main1}  Let  $(M,g)$ be an $n$-dimensional complete Riemannian manifold  and $S=\{x_1,...,x_m\}\subset M$ be the set of distinct poles, where $n\geq 3$ and $m\geq 2$.  
	Then 
	\begin{eqnarray}\label{originial-general}
	\int_{M} |\nabla_g u|^2\d&\geq&\frac{(n-2)^2}{m^2}\sum_{1\leq i<j\leq m}\int_{M} \left|\frac{\nabla_g d_i}{d_i}-\frac{\nabla_g d_j}{d_j}\right|^2u^2\d\nonumber\\&&+ \frac{n-2}{m}\sum_{i=1}^{m}\ds\int_{M} \frac{d_i\Delta_g d_i-(n-1)}{d_i^2}u^2\d,\ \ \forall u\in C_0^\infty(M).\end{eqnarray}
	Moreover, in the bipolar case $($i.e., $m=2$$)$, the constant $ \frac{(n-2)^2}{m^2}=\frac{(n-2)^2}{4}$ is optimal in {\rm (\ref{originial-general})}.
\end{theorem}


\begin{remark}\rm \label{remark-elso}
	(a) The proof of inequality (\ref{originial-general}) is based on  a direct calculation. If $m=2$, the local behavior of geodesic balls implies the  optimality of the constant $ \frac{(n-2)^2}{m^2}=\frac{(n-2)^2}{4}$; in particular,  the second term is a lower order perturbation of the first one of the RHS (independently of the curvature). 
	
	(b) The optimality of $\frac{(n-2)^2}{m^2}$ seems to be a hard nut to crack.  
	A possible approach could be a fine Agmon-Allegretto-Piepenbrink-type  spectral estimate developed by Devyver \cite{D} and Devyver, Fraas and Pinchover \cite{DFP} whenever $(M,g)$ has 	asymptotically non-negative Ricci curvature (see Pigola, Rigoli and Setti  \cite[Corollary 2.17, p. 44]{Pigola-Rigoli-Setti2008}). Indeed, under this curvature assumption one can prove that the operator $-\Delta_g-W$ is critical (see \cite[Definition 4.3]{DFP}), where 
	$$W=\frac{(n-2)^2}{m^2}\sum_{1\leq i<j\leq m} \left|\frac{\nabla_g d_i}{d_i}-\frac{\nabla_g d_j}{d_j}\right|^2+\frac{n-2}{m}\sum_{i=1}^{m}\ds \frac{d_i\Delta_g d_i-(n-1)}{d_i^2}.$$
	Although expected,  we have no full control on the second summand with respect to the first one in $W$, i.e., the latter term could compete with the 'leading' one; clearly, in the Euclidean setting no such competition is present, thus the optimality of $\frac{(n-2)^2}{m^2}$ immediately follows by the criticality of $W$. It remains to investigate this issue in a forthcoming study. 
	
	
%
	
(c)  We emphasize that the second term in the RHS of (\ref{originial-general}) has a crucial role. Indeed, on one hand, when the Ricci curvature verifies Ric$(M,g)\geq c_0(n-1)g$ for some $c_0>0$, one has that $d_i(x)=g_d(x,x_i)\leq \pi/\sqrt{c_0}$ for every $x\in M$ and by the Laplace comparison theorem, we have that $d_i\Delta_g d_i-(n-1)\leq (n-1)(\sqrt{c_0}d_i\cot(\sqrt{c_0}d_i)-1)< 0$ for $d_i>0$, i.e. for every $x\neq x_i$. Thus, this term modifies the original problem (\ref{natural-MHIEUC}) by filling the gap in a suitable way. On the other hand, when $(M,g)$ is a Cartan-Hadamard manifold, one has $d_i\Delta_g d_i-(n-1)\geq 0$, and inequality (\ref{originial-general}) implies (\ref{natural-MHIEUC}). This result will be resumed in Corollary \ref{corollary-extension} (i).   In particular, when $M=\R^n$ is the Euclidean space, then $\exp_x(y)=x+y$ for every $x,y\in \mathbb R^n$ and $|x|\Delta |x|=n-1$ for every $x\neq 0$; therefore, Theorem \ref{main-main1} and the criticality of $-\Delta-W$ 
immediately yield the main result of Cazacu and Zuazua \cite{cazacuzuazua}, i.e., inequality (\ref{paramultipolar}) (and equivalently (\ref{MHIEUC})).

\end{remark}

	For further use, we notice that ${\bf K}\geq {c}$ (resp. ${\bf K}\leq {c}$) means that the sectional curvature on $(M,g)$ is
bounded from below (resp. above) by ${c}\in \mathbb R$ at any point and direction.

	For every $c\in \mathbb R$, let ${\bf s}_{c},{\bf
		ct}_{c}:[0,\infty)\to \mathbb
	R$ be defined by
	\begin{equation}\label{kell}
	{\bf s}_{c}(r)=\left\{
	\begin{array}{lll}
	\ds\frac{\sin(\sqrt{c}r)}{\sqrt{c}} & \hbox{if} & {c}>0,\\
	r
	& \hbox{if} &  {c}=0, \\
	\ds\frac{\sinh(\sqrt{-c}r)}{\sqrt{-c}} & \hbox{if} & {c}<0,
	\end{array}\right.\ \ \ {\rm and}\ \ \ {\bf ct}_{c}(r)=\left\{
	\begin{array}{lll}
	\sqrt{c}\cot(\sqrt{c}r) & \hbox{if} & {c}>0,\\
	\frac{1}{r}
	& \hbox{if} &  {c}=0, \\
	\sqrt{-c}\coth(\sqrt{-c}r) & \hbox{if} & {c}<0.
	\end{array}\right.
	\end{equation}


Although the paralelogrammoid law in the Euclidean setting provides the equivalence between (\ref{MHIEUC}) and (\ref{paramultipolar}), this property is no longer valid on generic manifolds. However, a curvature-based quantitative  paralelogrammoid law and a Toponogov-type comparison result provide a suitable counterpart of inequality (\ref{MHIEUC}):  


\begin{theorem}[\textbf{Multipolar Hardy inequality II}]\label{masodik}
	Let  $(M,g)$ be an $n$-dimensional complete Riemannian manifold with  ${\bf K}\geq k_0$ for some $k_0\in \mathbb R$ and let $S=\{x_1,...,x_m\}\subset M$ be the set of distinct poles belonging to a strictly convex open set $\tilde S\subset M$, where $n\geq 3$ and $m\geq 2$.  
Then  we have the following inequality: 
\begin{align}\label{multipolarcosinegeneral}
\int_{\tilde S} |\nabla_g u|^2\d\geq& \frac{4(n-2)^2}{m^2}\sum_{1\leq i<j\leq m}\int_{\tilde S} \frac{{\bf s}_{k_0}^2\left(\frac{d_{ij}}{2}\right)}{d_id_j{\bf s}_{k_0}(d_{i}){\bf s}_{k_0}(d_{j})}u^2\d+\sum_{1\leq i<j\leq m} \int_{\tilde S} R_{ij}(k_0)u^2\d \nonumber\\&+\frac{n-2}{m}\sum_{i=1}^m\int_{\tilde S} \frac{d_i\Delta_g d_i-(n-1)}{d_i^2}u^2\d,\ \ \forall u\in C_0^\infty(\tilde S),
\end{align}
where $d_{ij}=d_g(x_i,x_j)$ and $${R}_{ij}(k_0)=\left\{
\begin{array}{lll}
\frac{1}{d_i^2}+\frac{1}{d_j^2}-\frac{2}{k_0d_id_j}\left(\frac{1}{\ds {\bf s}_{k_0}(d_i){\bf s}_{k_0}(d_j)}-{\bf ct}_{k_0}(d_i){\bf ct}_{k_0}(d_j)\right), & \hbox{\rm if }& k_0\neq 0,\\
0, & \hbox{\rm if }& k_0=0. 
\end{array}
\right.$$	
\end{theorem}

\begin{remark}\rm \label{remark-masik-optimalitas}
When $(M,g)$ is a Cartan-Hadamard manifold and $k_0\leq 0$, one has that  $R_{ij}(k_0)\geq 0$; thus we obtain a similar result as in (\ref{natural-MHIEUC}); the precise statement will be given in Corollary \ref{corollary-extension} (ii).
\end{remark}

\medskip

\medskip

\noindent {\it Applications.} 
As we already noticed, multipolar Hardy inequalities have been applied in the flat case to  guaranty existence and uniqueness of solutions for various elliptic PDEs. If  the particles (e.g. the fermions appearing in the Thomas-Fermi theory, see Lieb \cite{Lieb}) are distributed in a curved space, the aforementioned works cannot be  applied. 
For instance, if some external forces perturb the flat model (present as a magnetic or gravitational field), the curvature will appear. 
Such a typical case occurs in the study of  classical particles in the Lobachevsky hyperbolic model or spherical Riemannian model, described recently by Kudryashov,  Kurochkin,  Ovsiyuk  and Red'kov \cite{oroszok}, and Cari\~nena,  Ra\~nada and Santander \cite{spanyolok}.  

Motivated by the latter investigations on curved frameworks, we consider two model Schr\"odinger-type equations involving  bipolar potentials in two different geometrical settings, namely, in the negatively and positively curved case, where our multipolar Hardy inequalities can be successfully applied:\\

A. {\it Non-positively curved case}. Let $(M,g)$ be an $n(\geq 3)$-dimensional Cartan-Hadamard manifold with ${\bf K}\geq k_0$ for some $k_0\leq 0$, and $S=\{x_1,x_2\}\subset M$ be the set of poles. 
By keeping the previous notations,  we consider the problem 
$$
-\Delta_g u +V(x)u=
\lambda \frac{{\bf s}_{k_0}^2\left(\frac{d_{12}}{2}\right)}{d_1d_2{\bf s}_{k_0}(d_{1}){\bf s}_{k_0}(d_{2})}u+\mu W(x)f(u)\ \   \hbox{ in } M, 
\eqno{(\mathscr{P}_M^\mu)}
$$
where    $V,W:M\to \R$  are positive potentials, $\lambda \in \left[0,{(n-2)^2}\right)$ is fixed, $\mu\geq 0$ is  a parameter, and the continuous function $f:\mathbb R\to \mathbb R$ is sublinear at infinity. In Theorem \ref{theorem-application-Hadamard} we prove that problem $(\mathscr{P}_M^\mu)$ has only the zero solution for small values of $\mu$, while it exists $\mu_0>0$ such that $(\mathscr{P}_M^\mu)$  has two distinct weak solutions in a suitable functional space whenever $\mu\geq \mu_0.$  \\

B. {\it Positively curved case}. If ${\mathbb{S}}_+^{n}$ denotes the open upper hemisphere and $S=\{x_1,x_2\}\subset {\mathbb{S}}_+^{n}$ is the set of poles, we study the Dirichlet problem 
	$$
	\left\{
	\begin{array}{ll}
	\ds-\Delta_{g} u +\mathtt{C}(n,\beta)u=
	\lambda \left|\frac{\nabla_{g} d_1}{d_1}-\frac{\nabla_{g} d_2}{d_2}\right|^2u+|u|^{p-2}u, & \hbox{ in } \mathbb{S}_+^{n} \\
	u=0, & \hbox{ on } \partial\mathbb{S}_+^{n},
	\end{array}
	\right.\eqno{(\mathscr{P}_{\mathbb{S}_+^{n}})}
	$$
	where $g$ is the usual Riemannian structure  on the  unit sphere $\mathbb S^n$ inherited by $\mathbb R^{n+1}$, $\lambda \in \left[0,\frac{(n-2)^2}{4}\right)$ is fixed, $\mathtt{C}(n,\beta)>0$ is given in Corollary \ref{corollary-ezleszajoagombon} and 
	$p\in (2,2^*)$;  hereafter, $2^*=2n/(n-2)$ is the critical Sobolev exponent.   In Theorem \ref{theorem-application-gomb} we prove the existence of infinitely many solutions for $(\mathscr{P}_{\mathbb{S}_+^{n}})$; moreover, by using group-theoretical arguments, we provide qualitative results on the solutions concerning their symmetries whenever  the poles $x_1$ and $x_2$ are in specific positions.

The plan of the paper is as follows. In \S \ref{fogalmak}  we present a series of preparatory definitions and results which are used throughout the paper. In \S \ref{section-2} we prove the multipolar Hardy inequalities, i.e.,  Theorems \ref{main-main1} \& \ref{masodik}.  
In \S \ref{section-appl}  we study  problems $(\mathscr{P}_M^\mu)$ and $(\mathscr{P}_{\mathbb{S}_+^{n}})$, while in \S \ref{section-final} we formulate some remarks concerning further questions/perspectives. 

\section{Preliminaries}\label{fogalmak}
Let $(M, g)$ be an $n-$dimensional complete
Riemannian manifold  $(n\geq 3)$. As usual,  $T_x
M$ denotes the tangent space at $x \in M$ and  $\displaystyle TM =
\bigcup_{x\in M}T_xM$ is the tangent bundle. Let $d_g : M \times M
\to [0, \infty)$ be the distance function associated to the
Riemannian metric $g$, and $B_r(x) = \{y \in M : d_g(x, y) <
r \}$ be the open geodesic ball with center $x\in M$ and radius $r
> 0$. If $\d$ is the canonical volume element on $(M, g)$, the
volume of a bounded open set $S \subset M$ is $\mathrm{Vol}_g(S) =
\displaystyle\int_S \d$. 
The behaviour of the volume of small geodesic balls can be expressed as follows, see Gallot,  Hulin  and Lafontaine \cite{GallotHulinLafontain}; for every $x\in M$ we have \begin{equation}\label{kicsigombok}
\mathrm{Vol}_g(B_r(x))=\omega_n r^n\left(1+o(r)\right) \hbox{ as }r \to 0.
\end{equation} 

Let $u:M\to \mathbb R$ be a function of
class $C^1.$ If $(x^i)$ denotes the local coordinate system on a
coordinate neighbourhood of $x\in M$, and the local components of
the differential of $u$ are denoted by $u_i=\frac{\partial
	u}{\partial
	x_i}$, then the local components of the gradient  $\nabla_g u$ are
$u^i=g^{ij}u_j$. Here, $g^{ij}$ are the local components of
$g^{-1}=(g_{ij})^{-1}$. In particular, for every $x_0\in M$ one has
the eikonal equation
\begin{equation}\label{dist-gradient}
|\nabla_g d_g(x_0,\cdot)|=1\ {\rm  a.e. \ on}\ M.
\end{equation}
In fact, relation (\ref{dist-gradient}) is valid for every point $x\in M$ outside of the cut-locus of $x_0$ (which is a null measure set). 

When no
confusion arises, if $X,Y\in T_x M$, we simply write $|X|$ and
$\langle X,Y\rangle$ instead of the norm $|X|_x$ and inner product
$g_x(X,Y)=\langle X,Y\rangle_x$, respectively.
The $L^p(M)$ norm of
$\nabla_g u(x)\in T_xM$ is given by
$$\|\nabla_g u\|_{L^p(M)}=\left(\displaystyle\int_M
|\nabla_gu|^p{\rm d}v_g\right)^{{1}/{p}}.$$  The space $H^1_g(M)$
is the completion of $C_0^\infty(M)$ with respect to the norm
$$\|u\|_{H^1_g(M)}=\sqrt{\|u\|_{L^2(M)}^2+\|\nabla_g u\|_{L^2(M)}^2}.$$
The Laplace-Beltrami operator is given by $\Delta_g u={\rm
	div}(\nabla_g u)$ whose expression in a local chart of associated
coordinates $(x^i)$ is $\Delta_g u=g^{ij}\left(\frac{\partial^2
	u}{\partial x_i\partial x_j}-\Gamma_{ij}^k\frac{\partial u}{\partial
	x_k}\right),$ where $\Gamma_{ij}^k$ are the coefficients of the
Levi-Civita connection.
%


In the sequel, we shall explore the following comparison results (see
Shen \cite{Shen-volume}, Wu and Xin \cite[Theorems 6.1 \&
6.3]{Wu-Xin}, Pigola, Rigoli and Setti  \cite[Theorem 2.4]{Pigola-Rigoli-Setti2008}):
\begin{itemize}
	\item {\it Laplace comparison theorem I}: if ${\bf
		K}\leq c$ for some $c\in \mathbb R$, then \begin{equation}  \label{laplace-comp-altalanos-0}
	\Delta_gd_g(x_0,x)\geq (n-1){\bf
		ct}_c(d_g(x_0,x));
	\end{equation}
	
		\item {\it Laplace comparison theorem II}: if $ {\bf K}\geq k_0$ for some $k_0\in \mathbb R$, then \begin{equation}  \label{laplace-comp-altalanos-11}
		\Delta_gd_g(x_0,x)\leq (n-1){\bf
			ct}_{k_0}(d_g(x_0,x)),
		\end{equation}
\end{itemize}
where these relations are understood in the distributional sense. Note that in (\ref{laplace-comp-altalanos-11}) it is enough to have the lower bound $(n-1)k_0$ for the Ricci curvature. 

\section{Multipolar Hardy inequalities: proof of Theorems \ref{main-main1} and \ref{masodik}}\label{section-2}

 {\it Proof of Theorem \ref{main-main1}.} Let $\ds E=\prod_{i=1}^m d_i^{2-n}$ and fix $u\in C_0^\infty(M)$ arbitrarily. 
	A direct calculation on the set  $M\setminus \bigcup_{i=1}^m (\{x_i\}\cup {\rm cut}(x_i))$ yields that
	$$\nabla_g\left(uE^{-\frac{1}{m}}\right)=E^{-\frac{1}{m}}\nabla_g u+  \frac{n-2}{m}u E^{-\frac{1}{m}}\sum_{i=1}^m\frac{\nabla_g d_i}{d_i}.$$
	Integrating the latter relation, the divergence theorem and eikonal equation (\ref{dist-gradient}) give that		\begin{align*}\int_{M} \left|\nabla_g\left(uE^{-\frac{1}{m}}\right)\right|^2 E^\frac{2}{m}\d=&\int_{M} |\nabla_g u|^2\d+\frac{(n-2)^2}{m^2} \int_{M}  \left|\sum_{i=1}^m\frac{\nabla_g d_i}{d_i}\right|^2u^2\d\\&+\frac{n-2}{m}\sum_{i=1}^m\int_{M} \left\langle\nabla_g u^2,\frac{\nabla_g d_i}{d_i} \right\rangle \d \label{baloldaliszamolas}\\ =&\int_{M} |\nabla_g u|^2\d+\frac{(n-2)^2}{m^2} \int_{M}  \left|\sum_{i=1}^m\frac{\nabla_g d_i}{d_i}\right|^2u^2\d\\&-\frac{n-2}{m}\sum_{i=1}^m\int_{M}  {\rm div}\left(\frac{\nabla_g d_i}{d_i}\right)u^2\d \nonumber.
	\end{align*}
	Due to (\ref{dist-gradient}), we have  $${\rm div}\left(\frac{\nabla_g d_i}{d_i}\right)=\frac{d_i\Delta_g d_i-1}{d_i^2},\ i\in \{1,...,m\}.$$
	Thus, an algebraic reorganization of the latter relation provides an  Agmon-Allegretto-Piepenbrink-type multipolar representation   \begin{eqnarray}\label{identity}
	 \nonumber \int_{M} |\nabla_g u|^2\d-\frac{(n-2)^2}{m^2}\sum_{1\leq i<j\leq m}\int_{M} \left|\frac{\nabla_g d_i}{d_i}-\frac{\nabla_g d_j}{d_j}\right|^2u^2\d&=&\int_{M} \left|\nabla_g\left(uE^{-1/m}\right)\right|^2 E^{2/m}\d\\&&+\frac{n-2}{m}\sum_{i=1}^m\mathcal K_i(u),\end{eqnarray}
	where $\ds \mathcal K_i(u)=\int_{M}\frac{d_i\Delta_g d_i-(n-1)}{d_i^2}u^2\d$. Inequality  (\ref{originial-general}) directly  follows by (\ref{identity}).  
	
In the sequel, we deal with the optimality of the constant $\frac{(n-2)^2}{m^2}$ in (\ref{originial-general}) when $m=2$. In this case the right hand side of (\ref{originial-general}) behaves as $\frac{(n-2)^2}{4}d_g(x,x_i)^{-2}$ whenever $x\to x_i$ and by the local behavior of the geodesic balls (see (\ref{kicsigombok})) we may expect the optimality of $\frac{(n-2)^2}{4}$. In order to be more explicit,  let $A_i[r,R]=\{x\in M:r\leq d_i(x)\leq R\}$  for $r<R$ and $i\in \{1,...,m\}$. 
If $0<r<<R$ are within the range of (\ref{kicsigombok}), a layer cake representation yields for every $i\in \{1,...,m\}$ that 
\begin{eqnarray}\label{distance-layer-cake}
\nonumber \int_{A_i[r,R]}d_i^{-n}\d&=&\frac{{\rm Vol}_g(B_R(x_i))}{R^n}-\frac{{\rm Vol}_g(B_r(x_i))}{r^{n}}+n\int_{r}^R{\rm Vol}_g( B_\rho(x_i))\rho^{-1-n}{\rm d}\rho\\
&=& o(R)+n\omega_n \log\frac{R}{r}.
\end{eqnarray}

%
%
%

Let $S=\{x_1,x_2\}$ be the set of poles, $x_1\neq x_2$.   
	 Let $\eps\in (0,1)$ be small enough such that it belongs to the range of (\ref{kicsigombok}),  
	 and $B_{2\sqrt{\eps}}(x_1)\cap B_{2\sqrt{\eps}}(x_2)=\emptyset$. 
%
	Let  $$u_\eps(x)=\left\{
	\begin{array}{lll}
	\frac{\log\left(\frac{d_i(x)}{\eps^2}\right)}{\log\left(\frac{1}{\eps}\right)}d_i(x)^\frac{2-n}{2}, & {\rm if} & x\in A_i[\eps^2,\eps]; \\
	\frac{2\log\left(\frac{\sqrt{\eps}}{d_i(x)}\right)}{\log\left(\frac{1}{\eps}\right)}d_i(x)^\frac{2-n}{2}, & {\rm if} & x\in A_i[\eps,\sqrt{\eps}];  \\
	0, & &\hbox{ otherwise,}
	\end{array}
	\right.
	$$
	with $i\in \{1,2\}.$
	Note that  $u_\eps\in C^0(M)$,  having compact support $\bigcup_{i=1}^2 A_i[\eps^2,\sqrt{\eps}]\subset M;$ in fact,  $u_\eps$ can be used as  a test function in (\ref{originial-general}). For later use let us denote by $\eps^*=\frac{1}{\log\left(\frac{1}{\eps}\right)^2},$ $$\ds\mathcal{I}_\eps=\int_{M} |\nabla_g u_\eps|^2\d, \  \ds\mathcal{L}_\eps=\int_{M}  \frac{\langle \nabla_g d_1,\nabla_g d_2 \rangle}{d_1d_2}u_\eps^2\d, \ \ds\mathcal K_\eps=\sum_{i=1}^{2}\int_{M}  \frac{d_i \Delta_g d_i-(n-1)}{d_i^2}u_\eps^{2}\d$$ and $$\ \ds\mathcal{J}_\eps=\int_{M}  \left[\frac{1}{d_1^2}+\frac{1}{d_2^2}\right]u_\eps^2\d.$$
	The proof is based on the following claims:   
	 	\begin{equation}\label{osszefugges-I-J}
	 	\mathcal{I}_\eps-\mu_H\mathcal{J}_\eps=\mathcal O(1),\ \mathcal{L}_\eps=\mathcal O(\sqrt[4]{\eps})\ and\ \mathcal{K}_\eps=\mathcal O(\sqrt[4]{\eps})\ \ {\rm as}\ \ \eps\to 0,
	 	\end{equation}
	 	and 
	 	\begin{equation}\label{hatarertek-J}
	 	\lim_{\eps\to 0}\mathcal{J}_\eps=+ \infty. 
	 	\end{equation}
	The above properties can be obtained by direct computations, based on  the estimates (\ref{kicsigombok}), (\ref{distance-layer-cake}) and 
	 	   $$\left|\Delta_g d_i-\frac{n-1}{d_i}\right|\leq 1 \mbox{ \ a.e.  in } B_{\sqrt{\varepsilon}}(x_i),$$
	 	(for $\eps>0$ small enough), see Krist\'aly and Repovs \cite{Kristaly-Repovs-2016-CCM}. 
Combining  relations (\ref{osszefugges-I-J}) and (\ref{hatarertek-J}) with inequality (\ref{originial-general}), we have that $$\mu_H\leq\frac{\mathcal{I}_\eps-\frac{n-2}{2}\mathcal{K}_\eps}{\mathcal{J}_\eps-2\mathcal{L}_\eps} \leq\frac{\mathcal{I}_\eps+\frac{n-2}{2}|\mathcal{K}_\eps|}{\mathcal{J}_\eps-2|\mathcal{L}_\eps|}=
  \frac{\mu_H\mathcal J_\eps+\mathcal O(1)}{\mathcal{J}_\eps+\mathcal O(\sqrt[4]{\eps})}\to \mu_H\ {\rm as}\ \eps\to 0,$$ 
  which concludes the proof. 
\hfill  $ \square$

\begin{remark}\rm
	\noindent 
	 Let us assume that in Theorem \ref{main-main1}, $(M,g)$ is a Riemannian manifold  with sectional curvature verifying ${\bf K}\leq c$.  
	By the Laplace comparison theorem I (see (\ref{laplace-comp-altalanos-0})) we have:
	\begin{eqnarray}\label{ohadamard-general}
	\int_{M} |\nabla_g u|^2\d&\geq&\frac{(n-2)^2}{m^2}\sum_{1\leq i<j\leq m}\int_{M} \left|\frac{\nabla_g d_i}{d_i}-\frac{\nabla_g d_j}{d_j}\right|^2u^2\d\nonumber\\&&+ \frac{(n-2)(n-1)}{m}\sum_{i=1}^{m}\int_{M} \frac{\boldsymbol{\rm D}_{c}(d_{i})}{d_i^2}u^{2}\d,\ \ \forall u\in C_0^\infty(M),\end{eqnarray}
	where $\boldsymbol{\rm D}_{c}(r)=r{\bf ct}_c(r)-1,$ $r\geq 0.$
	In addition, if $(M,g)$ is a Cartan-Hadamard manifold with ${\bf K}\leq c\leq 0$, then 
	  $\boldsymbol{\rm D}_{c}(r)\geq \frac{3|c|r^2}{\pi^2+|c|r^2}$ for all $r\geq 0.$
	Accordingly, stronger curvature of the Cartan-Hadamard manifold implies  improvement in the multipolar Hardy inequality (\ref{ohadamard-general}). 
\end{remark}


 {\it Proof of Theorem \ref{masodik}.}
It is clear that 
\begin{equation}\label{skalar-szorzat}
\left|\frac{\nabla_g d_i}{d_i}-\frac{\nabla_g d_j}{d_j}\right|^2=\frac{1}{d_i^2}+\frac{1}{d_j^2}-2\frac{\langle \nabla_g d_i,\nabla_g d_j \rangle}{d_id_j}.
\end{equation}
Let us fix two arbitrary poles $x_i$ and $x_j$ $(i\neq j)$, and a point $x\in \tilde S$. We consider the Alexandrov comparison triangle with vertexes $\tilde x_i$, $\tilde x_j$ and $\tilde x$ in the space $M_0$ of constant sectional curvature $k_0$, associated to the points $x_i$, $x_j$ and $x$, respectively. More precisely, $M_0$ is the $n$-dimensional hyperbolic space of curvature $k_0$ when $k_0<0$, the Euclidean space when $k_0=0$, and the sphere with curvature $k_0$ when $k_0>0.$

We first prove that the perimeter $L(x_ix_jx)$ of the geodesic triangle $x_ix_jx$ is strictly less than $\frac{2\pi}{\sqrt{k_0}}$; clearly, when $k_0\leq 0$ we have nothing to prove. Due to the strict convexity of $\tilde S$, the unique geodesic segments joining pairwisely  the points  $x_i$, $x_j$ and $x$ belong entirely to $\tilde S$ and as such, these points are not conjugate to each other. Thus, due to do Carmo \cite[Proposition 2.4, p. 218]{doCarmo}, every side of the geodesic triangle has length  $\leq \frac{\pi}{\sqrt{k_0}}.$ By Klingenberg  \cite[Theorem 2.7.12, p. 226]{Klingenberg} we have that $L(x_ix_jx)\leq \frac{2\pi}{\sqrt{k_0}}$. Moreover, by the same result of Klingenberg, if $L(x_ix_jx)= \frac{2\pi}{\sqrt{k_0}}$, it follows that either $x_ix_jx$ forms a closed geodesic, or $x_ix_jx$ is a geodesic biangle (one of the sides has length $\frac{\pi}{\sqrt{k_0}}$ and the two remaining sides form together  a minimizing geodesic of length $\frac{\pi}{\sqrt{k_0}}$). In both cases we find points on the sides of the geodesic triangle $x_ix_jx$ which can be joined by two minimizing geodesics,  contradicting the strict convexity of $\tilde S$.


 We are now in the position to apply a Toponogov-type comparison result, see Klingenberg  \cite[Proposition 2.7.7, p. 220]{Klingenberg}; namely, we have the comparison of angles
$$\gamma_{M_0}=m\widehat{(\tilde{x}_i\tilde{x}\tilde{x_j})}\leq \gamma_M=m\widehat{(x_ixx_j)}.$$ Therefore, $\langle \nabla_g d_i,\nabla_g d_j \rangle=\cos(\gamma_M)\leq \cos(\gamma_{M_0}).$ 

 On the other hand, by the cosine-law on the space form $M_0$, see Bridson and Haefliger \cite[p. 24]{Bridson-Haefliger}, we have 
$$\left\{
\begin{array}{lll}
	\cosh(\sqrt{-k_0}d_{ij})=\cosh(\sqrt{-k_0}d_{i})\cosh(\sqrt{-k_0}d_{j})-\sinh(\sqrt{-k_0}d_{i})\sinh(\sqrt{-k_0}d_{j})\cos(\gamma_{M_0}), & {\rm if} & \hbox{\rm $k_0<0;$} \\
	\cos(\sqrt{k_0}d_{ij})=\cos(\sqrt{k_0}d_{i})\cos(\sqrt{k_0}d_{j})+\sin(\sqrt{k_0}d_{i})\sin(\sqrt{k_0}d_{j})\cos(\gamma_{M_0}), & {\rm if} & \hbox{\rm $k_0>0$;} \\
	d_{ij}^2=d_i^2+d_j^2-2d_id_j\cos(\gamma_{M_0}), & {\rm if} & \hbox{\rm $k_0=0$.}
	\end{array}
\right.$$
Consequently, $$\left\{
\begin{array}{lll}
	\cos(\gamma_{M})\leq\frac{\cosh(\sqrt{-k_0}d_{i})\cosh(\sqrt{-k_0}d_{j})-\cosh(\sqrt{-k_0}d_{ij})}{\sinh(\sqrt{-k_0}d_{i})\sinh(\sqrt{-k_0}d_{j})}, & {\rm if} & \hbox{\rm $k_0<0;$} \\
	\cos(\gamma_M)\leq\frac{\cos(\sqrt{k_0}d_{ij})-\cos(\sqrt{k_0}d_{i})\cos(\sqrt{k_0}d_{j})}{\sin(\sqrt{k_0}d_{i})\sin(\sqrt{k_0}d_{j})}, & {\rm if} & \hbox{\rm $k_0>0;$}  \\
	\cos(\gamma_M)\leq \frac{d_i^2+d_j^2-d_{ij}^2}{2d_id_j}, & {\rm if} & \hbox{\rm $k_0=0,$}
\end{array}
\right.$$
which implies 
$$\frac{1}{d_i^2}+\frac{1}{d_j^2}-\frac{2\cos(\gamma_M)}{d_id_j}\geq \left\{
\begin{array}{lll}
\frac{4}{d_id_j}\frac{{\bf s}_{k_0}^2\left(\frac{d_{ij}}{2}\right)}{{\bf s}_{k_0}(d_{i}){\bf s}_{k_0}(d_{j})}+R_{ij}(k_0), & {\rm if} & \hbox{\rm $k_0\neq 0;$} \\
\frac{d_{ij}^2}{d_i^2d_j^2}, & {\rm if} & \hbox{\rm $k_0=0,$}
\end{array}
\right.$$
where the expression $R_{ij}(k_0)$ is given in the statement of the theorem. Relation (\ref{skalar-szorzat}), the above inequality and (\ref{originial-general}) imply together \eqref{multipolarcosinegeneral}. \hfill $\square$\\

\section{Applications: bipolar Schr\"odinger-type equations on curved settings}\label{section-appl}

In this section we present two applications in different geometric frameworks. In order to avoid technicalities, we shall restrict our attention to problems with only two poles; the interested reader may extend these results to multiple poles with suitable modifications.

\subsection{A bipolar Schr\"odinger-type equation on Cartan-Hadamard manifolds}

First of all, by using inequalities (\ref{originial-general}) and (\ref{multipolarcosinegeneral}),  we obtain the following non-positively curved versions of Cazacu and Zuazua's  inequalities (\ref{paramultipolar}) and (\ref{MHIEUC}) for multiple poles, respectively:

\begin{corollary}\label{corollary-extension}	Let $(M,g)$ be an $n$-dimensional Cartan-Hadamard  manifold and let $S=\{x_1,...,x_m\}\subset M$ be the set of distinct poles, with  $n\geq 3$ and $m\geq 2$. Then we have the following inequality:
	\begin{equation}\label{original_on_H1}
	\int_{M} |\nabla_g u|^2\d\geq\frac{(n-2)^2}{m^2}\sum_{1\leq i<j\leq m}\int_{M} \left|\frac{\nabla_g d_i}{d_i}-\frac{\nabla_g d_j}{d_j}\right|^2u^2\d,\ \ \forall u\in H^1_g(M).\end{equation}
	Moreover, if ${\bf K}\geq k_0$ for some $k_0\in \mathbb R,$
	then 
	\begin{equation}\label{Hadamardcosine}
	\int_M |\nabla_g u|^2\d\geq\frac{4(n-2)^2}{m^2}\sum_{1\leq i<j\leq m}\int_M \frac{{\bf s}_{k_0}^2\left(\frac{d_{ij}}{2}\right)}{d_id_j{\bf s}_{k_0}(d_{i}){\bf s}_{k_0}(d_{j})}u^2\d,\ \ \forall u\in H^1_g(M).
	\end{equation}
\end{corollary}

{\it Proof}. Since $(M,g)$ is a Cartan-Hadamard manifold, by using inequality (\ref{originial-general}) and the Laplace comparison theorem I (i.e., inequality (\ref{laplace-comp-altalanos-0}) for $c=0$), standard approximation procedure based on the density of $C_0^\infty(M)$ in $H_g^1(M)$ and Fatou's lemma immediately imply (\ref{original_on_H1}). Moreover, elementary properties of hyperbolic functions show that $R_{ij}(k_0)\geq 0$ (since $k_0\leq 0$). Thus, the latter inequality and  (\ref{multipolarcosinegeneral}) yield  (\ref{Hadamardcosine}). 
\hfill $\square$ \\

In the sequel, let $(M,g)$ be an $n$-dimensional Cartan-Hadamard manifold $(n\geq 3)$ with ${\bf K}\geq k_0$ for some $k_0\leq 0$, and $S=\{x_1,x_2\}\subset M$ be the set of poles.  In this subsection we deal with the Schr\"odinger-type equation
$$
-\Delta_g u +V(x)u=
\lambda \frac{{\bf s}_{k_0}^2\left(\frac{d_{12}}{2}\right)}{d_1d_2{\bf s}_{k_0}(d_{1}){\bf s}_{k_0}(d_{2})}u+\mu W(x)f(u)\ \   \hbox{ in } M, \\
\eqno{(\mathscr{P}_M^\mu)}
$$
where   $\lambda \in \left[0,{(n-2)^2}\right)$ is fixed, $\mu\geq 0$ is  a parameter, and the continuous function $f:[0,\infty)\to \mathbb R$ verifies 
\begin{itemize}
	\item[$(f_1)$] $f(s)=o(s)$ as $s\to 0^+$ and $s\to \infty;$
	\item[$(f_2)$] $F(s_0)>0$ for some $s_0> 0,$ where $\ds F(s)=\int_0^s f(t){\rm d}t.$
\end{itemize}
According to 
$(f_1)$ and  $(f_2)$,
the number $c_f=\max_{s>0}\frac{f(s)}{s}$ is well defined and positive.

On the potential $V:M\to \mathbb R$ we require that
\begin{itemize}
	\item[$(V_1)$] $\ds V_0=\inf_{x\in M}V(x)>0$;
	\item[$(V_2)$] $\ds \lim_{d_g(x_0,x)\to \infty}V(x)=+\infty$ for some $x_0\in M$,
\end{itemize}
and $W:M\to \mathbb R$ is assumed to be positive. 
Elliptic problems with similar assumptions on $V$ have been studied on Euclidean spaces, see e.g. Bartsch, Pankov and Wang \cite{Bartsch-Pankov-Wang}, Bartsch and Wang \cite{Bartsch-Wang}, Rabinowitz \cite{Rabinowitz-ZAMP} and Willem \cite{Willem-book}.


 Before to state our result, let us consider the functional space $$H^1_V(M)=\left\{u\in H^1_g(M): \im{\left(|\nabla_g u|^2+V(x)u^2\right)}<+\infty\right\}$$
endowed with the norm $$\|u\|_V=\left(\im{|\nabla_g u|^2}+\im{V(x)u^2}\right)^{1/2}.$$
The next Rabinowitz-type compactness result (see Rabinowitz \cite{Rabinowitz-ZAMP}) is crucial in the study of weak solutions of  problem $(\mathscr{P}_M^\mu)$:  
\begin{lemma}\label{Hadamardvegtelen-lemma} If $V$ satisfies $(V_1)$ and $(V_2)$,  the embedding $H^1_V(M)\hookrightarrow L^p(M)$ is compact, $p\in [2,2^*)$.
\end{lemma}
\begin{proof} Let $\{u_k\}_k \subset H^1_V(M)$ be a bounded sequence in $H^1_V(M)$, i.e., $\|u_k\|_V\leq \eta$ for some $\eta>0.$
 Let $q>0$ be arbitrarily fixed; by $(V_2)$, there exists $R>0$ such that $V(x)\geq q$ for every $x\in M\setminus B_R(x_0)$. 
	Thus,  $$\int_{M\setminus B_R(x_0)} (u_k-u)^2\d\leq \frac{1}{q}\int_{M\setminus B_R(x_0)}V(x)|u_k-u|^2\leq \frac{(\eta+\|u\|_V)^2}{q}.$$
	 On the other hand, 	by $(V_1)$, we have that $H^1_V(M)\hookrightarrow H_g^1(M)\hookrightarrow L_{\rm loc}^2(M)$; thus, up to a subsequence we have that $u_k \to u$ in $L^2_{\rm loc}(M)$. Combining the above two facts and taking into account that $q>0$ can be arbitrary large, we deduce that  $u_k\to u$ in $L^2(M)$; thus the embedding follows for $p=2$.  Now, if $p\in (2,2^*)$, by using an interpolation inequality and the Sobolev inequality on Cartan-Hadamard manifolds (see Hebey \cite[Chapter 8]{Hebey-book}), 
	 one has 
	 \begin{eqnarray*}
	 \|u_k-u\|_{L^p(M)}^p&\leq& \|u_k-u\|_{L^{2^*}(M)}^{n(p-2)/2}\|u_k-u\|_{L^2(M)}^{n(1-p/2^*)}\\&\leq& \mathcal C_n\|\nabla_g(u_k-u)\|_{L^{2}(M)}^{n(p-2)/2}\|u_k-u\|_{L^2(M)}^{n(1-p/2^*)},
	 \end{eqnarray*}
	 where $\mathcal C_n>0$  depends on $n$. Therefore,  
	  $u_k\to u$ in $L^p(M)$ for every $p\in (2,2^*)$. 
\end{proof}

The main result of this subsection is as follows.

\begin{theorem}\label{theorem-application-Hadamard}   
Let $(M,g)$ be an $n$-dimensional Cartan-Hadamard manifold $(n\geq 3)$ with ${\bf K}\geq k_0$ for some $k_0\leq 0$ and let $S=\{x_1,x_2\}\subset M$ be the set of distinct poles.   Let $V,W:M\to \mathbb R$ be positive potentials verifying $(V_1)$, $(V_2)$ and	$W\in L^1(M)\cap L^\infty(M)\setminus\{0\}$, respectively.  Let $f:[0,\infty)\to \mathbb R$  be a continuous
	function verifying $(f_1)$ and $(f_2)$, and $\lambda \in \left[0,{(n-2)^2}\right)$ be fixed.
	Then the following statements hold: 
	\begin{itemize}
		\item[{\rm (i)}]  Problem $(\mathscr{P}_M^\mu)$ has only the zero
		solution whenever  $0\leq
		\mu<V_0\|W\|_{L^\infty(M)}^{-1}c_f^{-1};$
		
		\item[{\rm (ii)}]  There exists $\mu_0>0$ such that problem $(\mathscr{P}_M^\mu)$  has at least
		two distinct  non-zero, non-negative weak
		solutions in $H_V^1(M)$ whenever
		$\mu>\mu_0$.
	\end{itemize}
\end{theorem}

 {\it Proof.} 
According to  $(f_1)$, one has $f(0)=0$. Thus, we may extend
 the function $f$ to the whole 
 $\mathbb R$ by $f(s)=0$ for $s\leq 0,$ which will be
 considered throughout the proof. Fix $\lambda \in \left[0,{(n-2)^2}\right)$.

 (i)  Assume that $u\in
 H_V^1(M)$ is a non-zero weak solution of problem $(\mathscr{P}_M^\mu)$.
 Multiplying $(\mathscr{P}_M^\mu)$ by  $u$, an integration on $M$ gives that 
 $$\im{|\nabla_g u|^2}+\im{V(x)u^2}=
 \lambda \int_M\frac{{\bf s}_{k_0}^2\left(\frac{d_{12}}{2}\right)}{d_1d_2{\bf s}_{k_0}(d_{1}){\bf s}_{k_0}(d_{2})}u^2\d+\mu \int_M W(x)f(u)u\d.$$
 By the latter relation, Corollary \ref{corollary-extension} (see relation (\ref{Hadamardcosine})) and the definition of  $c_f$, it yields that 
 \begin{eqnarray*}
 \im{|\nabla_g u|^2}+V_0\im{u^2} &\leq& \im{|\nabla_g u|^2}+\im{V(x)u^2}\\
 &=&
 \lambda \int_M\frac{{\bf s}_{k_0}^2\left(\frac{d_{12}}{2}\right)}{d_1d_2{\bf s}_{k_0}(d_{1}){\bf s}_{k_0}(d_{2})}u^2\d+\mu \int_M W(x)f(u)u\d\\&\leq&
 \im{|\nabla_g u|^2}+ \mu \|W\|_{L^\infty(M)}c_f\im{u^2}.
 \end{eqnarray*}
 Consequently, if $0\leq \mu
 <V_0\|W\|_{L^\infty(M)}^{-1}c_f^{-1},$ then 
 $u$ is necessarily 0, a contradiction. 
 
 (ii) 
Let us consider the energy functional associated with problem  $(\mathscr{P}_M^\mu)$, i.e., $\mathcal{E}_\mu: H^1_V(M)\to \R$ defined by  $$\mathcal{E}_\mu(u)=\frac{1}{2}\im{(|\nabla_g u|^2+V(x)u^2)}-\frac{\lambda}{2}\im{\frac{{\bf s}_{k_0}^2\left(\frac{d_{12}}{2}\right)}{d_1d_2{\bf s}_{k_0}(d_{1}){\bf s}_{k_0}(d_{2})}u^2}-\mu \int_M W(x)F(u)\d.$$ One can show that $\mathcal{E}_\mu\in C^1(H^1_V(M),\R)$ and  for all $u,w\in H^1_V(M)$ we have \begin{align*}\mathcal{E}_\mu'(u)(w)&=\im{(\langle \nabla_gu, \nabla_g w\rangle+V(x)uw)}-\lambda\im{\frac{{\bf s}_{k_0}^2\left(\frac{d_{12}}{2}\right)}{d_1d_2{\bf s}_{k_0}(d_{1}){\bf s}_{k_0}(d_{2})}uw}-\mu \int_M W(x)f(u)w\d.\end{align*}
Therefore, the critical points of $\mathcal{E}_\mu$ are precisely the weak solutions of  problem $(\mathscr{P}_M^\mu)$ in $H^1_V(M)$. By exploring the sublinear character of $f$ at infinity, Corollary \ref{corollary-extension} and Lemma \ref{Hadamardvegtelen-lemma},  one can see that $\mathcal{E}_\mu$ is  bounded from
below, coercive and satisfies the usual Palais-Smale condition for every $\mu\geq 0$.  Moreover, by  an elementary computation one can see that assumption $(f_1)$ is inherited as a sub-quadratic property in the sense that
\begin{equation}\label{sublinear-sobolev}
\lim_{		\|u\|_{V}\to 0}\frac{\ds\int_M W(x)F(u)\d}{\|u\|_{V}^2}=\lim_{		\|u\|_{V}\to \infty}\frac{\ds\int_M W(x)F(u)\d}{\|u\|_{V}^2}=0.
\end{equation}
Due to $(f_2)$ and $W\neq 0$, we can construct a non-zero truncation function $u_0\in
 H_V^1(M)$ such that $\ds\int_M W(x)F(u_0)\d>0.$ Thus,
 we may define
 $$\mu_0=\frac{1}{2}\inf\left\{\frac{\|u\|_{V}^2}{\ds\int_M W(x)F(u)\d}: u\in H_V^1(M), \ds\int_M W(x)F(u)\d>0
 \right\}.$$ By the relations in 
 (\ref{sublinear-sobolev}), we clearly have that $0<\mu_0<\infty.$ 
 
 Let us fix $\mu>\mu_0.$ Then there exists
 $\tilde u_\mu\in H_V^1(M)$ with $\ds\int_M W(x)F(\tilde u_\mu)\d>0$ such that $\mu>\frac{\ds\|\tilde u_\mu\|_{V}^2}{\ds 2\int_M W(x)F(\tilde u_\mu)\d}\geq\mu_0$.
 Consequently,  $$c_\mu^1:=\inf_{H_V^1(M)}\mathcal E_\mu\leq \mathcal E_\mu(\tilde u_\mu)\leq \frac{1}{2}\|\tilde u_\mu\|_V^2-\mu\ds\int_M W(x)F(\tilde u_\mu)< 0.$$
 Since $\mathcal E_\mu$ is bounded from below and satisfies the
 Palais-Smale condition, the number $c_\mu^1$ is a critical value
 of $\mathcal E_\mu$, i.e.,  there exists $u_\mu^1\in
 H_V^1(M)$ such that $\mathcal
 E_\mu(u_\mu^1)=c_\mu^1<0$ and $\mathcal
 E_\mu'(u_\mu^1)=0.$ In particular, $u_\mu^1\neq 0$ is a weak solution of problem $(\mathscr{P}_M^\mu)$.

Standard computations based on Corollary \ref{corollary-extension} and the embedding $H^1_V(M)\hookrightarrow L^p(M)$ for $p\in (2,2^*)$ show that there exists a sufficiently small $\rho_\mu\in (0,\|\tilde u_\mu\|_V)$ such that 
$$\inf_{\|u\|_{V}=\rho_\mu}\mathcal E_\mu(u)=\eta_\mu>0=\mathcal E_\mu(0)> \mathcal E_\mu(\tilde
u_\mu),$$ which means that the functional $\mathcal E_\mu$ has the
 mountain pass geometry. Therefore, we may apply
the mountain pass theorem, see Rabinowitz \cite{Rabinowitz-ZAMP}, showing that there exists $u_\mu^2\in
H_V^1(M)$ such that $\mathcal E_\mu'(u_\mu^2)=0$ and
$\mathcal E_\mu(u_\mu^2)=c_\mu^2$, where 
$c_\mu^2=\inf_{\gamma\in \Gamma}\max_{t\in [0,1]}\mathcal E_\mu(\gamma(t)),$
and $\Gamma=\{\gamma\in C([0,1];H_V^1(M)):\gamma(0)=0,\
\gamma(1)=\tilde u_\mu\}.$ Due to the fact that  $c_\mu^2\geq
\inf_{\|u\|_{V}=\rho_\mu}\mathcal E_\mu(u)>0$, it is
clear that $0\neq u_\mu^2\neq u_\mu^1.$ Moreover, since $f(s)=0$ for
every $s\leq 0,$ the solutions $u_\mu^1$ and $u_\mu^2$ are
non-negative.  
\hfill $\Box$

\begin{remark}\rm 
	Theorem \ref{theorem-application-Hadamard}    can be applied  on the hyperbolic space $\mathbb H^n=\{y=(y_1,...,y_n):y_n>0\}$ endowed with the metric
	$g_{ij}(y_1,...,y_n)=\frac{\delta_{ij}}{y_n^2}$; it is new even on the Euclidean space $\mathbb R^n$, $n\geq 3$. 
\end{remark}

\subsection{A bipolar Schr\"odinger-type equation on the upper hemisphere}

A positively curved counterpart of (\ref{original_on_H1}) can be stated as follows by using (\ref{originial-general}) and a Mittag-Leffler expansion  (the interested reader can establish a similar inequality to (\ref{Hadamardcosine}) as well):

\begin{corollary}\label{corollary-ezleszajoagombon}
	Let ${\mathbb{S}}_+^{n}$ be the open upper hemisphere and let  $S=\{x_1,...,x_m\}\subset {\mathbb{S}}_+^{n}$ be the set of distinct poles, with   $n\geq 3$ and $m\geq 2$. Let $\ds \beta=\max_{i=\overline{1,m}} d_g(x_0,x_i)$, where $x_0=(0,...,0,1)$ is the north pole of the sphere ${\mathbb{S}}^{n}$ and $g$ is the natural Riemannian metric of $\mathbb S^n$ inherited by $\mathbb R^{n+1}$.   Then we have the following inequality:
	\begin{equation}\label{csakezmegyjol}
	\|u\|^2_{\mathtt{C}(n,\beta)}\geq \frac{(n-2)^2}{m^2}\sum_{1\leq i<j\leq m}\ig{\left|\frac{\nabla_g d_i}{d_i}-\frac{\nabla_g d_j}{d_j}\right|^2u^2},\ \ \forall u\in H^1_g(\mathbb{S}_+^{n}),
	\end{equation}
	where $\ds\|u\|^2_{{\mathtt{C}(n,\beta)}}=\int_{\mathbb{S}_+^{n}}|\nabla_g u|^2\d+\mathtt{C}(n,\beta)\int_{\mathbb{S}_+^{n}} u^2\d$ and $\mathtt{C}(n,\beta)=(n-1)(n-2)\frac{7\pi^2-3\left(\beta+\frac{\pi}{2}\right)^2}{2\pi^2\left(\pi^2-\left(\beta+\frac{\pi}{2}\right)^2\right)}$.
\end{corollary}

{\it Proof}. Let $M=\mathbb S^n$ be the standard unit sphere in $\mathbb R^{n+1}$ and the open upper hemisphere $\mathbb S^n_+=\{y=(y_1,...,y_{n+1})\in \mathbb S^n:y_{n+1}>0\}.$ 
By Theorem \ref{main-main1} we have  \begin{align*}\ig{|\nabla_g u|^2}\geq& \frac{(n-2)^2}{m^2}\sum_{1\leq i<j\leq m}\igo{ \left|\frac{\nabla_g d_i}{d_i}-\frac{\nabla_g d_j}{d_j}\right|^2u^2}\\&+\frac{n-2}{m}\sum_{i=1}^{m}\ds\ig{ \frac{d_i\Delta_g d_i-(n-1)}{d_i^2}u^2},\ \forall u\in C_0^\infty(\mathbb S^n_+).\end{align*}
Since ${\bf K}\equiv 1$,  the  two-sided Laplace comparison theorem (or a direct computation) shows that $\Delta_g d_i=(n-1)\cot(d_i).$

Fix $u\in C_0^\infty(\mathbb S_+^n)$.  By  using the Mittag-Leffler expansion of the cotangent function, i.e., $$\cot t=\frac{1}{t}+2t\sum_{k=1}^\infty\frac{1}{t^2-\pi^2k^2},\ t\in (0,\pi),$$
and the fact that $0<d_i<\pi,$  $i\in \{1,...,m\}$ (up to the poles, which has null measure), one has 
$$\ds\ig{ \frac{d_i\Delta_g d_i-(n-1)}{d_i^2}u^2}=-2(n-1)\int_{S_+^n}\sum_{k=1}^\infty\frac{u^2}{\pi^2k^2-d_i^2}{\rm d}v_g.$$		
Since $d_i< \pi$, we get that $$\int_{{\mathbb{S}}_+^{n}}\sum_{k=2}^\infty\frac{u^2}{\pi^2k^2-d_i^2}\d\leq \int_{{\mathbb{S}}_+^{n}}\sum_{k=2}^\infty\frac{u^2}{\pi^2k^2-\pi^2}\d=\frac{3}{4\pi^2}\int_{{\mathbb{S}}_+^{n}}u^2\d.$$
Moreover, since $\ds \beta=\max_{i=\overline{1,m}}d_g(x_0,x_i)<\frac{\pi}{2},$  one can see that for every $x\in \Sg$, $d_i(x)=d_g(x,x_i)\leq d_g(x,x_0)+d_g(x_0,x_i)< \frac{\pi}{2}+\beta.$
Thus,  $\pi^2-d_i^2> \pi^2-\left(\beta+\frac{\pi}{2}\right)^2>0,$
which implies  $$\ig{\frac{u^2}{\pi^2-d_i^2}}\leq \frac{1}{\pi^2-\left(\beta+\frac{\pi}{2}\right)^2}\ig{u^2}.$$
Combining the above two estimates,  we have that $$\ig{|\nabla_g u|^2}+\mathtt{C}(n,\beta)\ig{u^2}\geq \frac{(n-2)^2}{m^2}\sum_{1\leq i<j\leq m}\ig{ \left|\frac{\nabla_g d_i}{d_i}-\frac{\nabla_g d_j}{d_j}\right|^2u^2},$$ where 
$\mathtt{C}(n,\beta)=(n-1)(n-2)\frac{7\pi^2-3\left(\beta+\frac{\pi}{2}\right)^2}{2\pi^2\left(\pi^2-\left(\beta+\frac{\pi}{2}\right)^2\right)}$.
The latter inequality can be extended to $H_g^1({\mathbb{S}}_+^{n})$ by standard approximation argument.\hfill $\square$\\

For simplicity, let $S=\{x_1,x_2\}\in {\mathbb{S}}_+^{n}$ be the set of poles. We consider the Dirichlet problem 
$$
\left\{
\begin{array}{ll}
\ds-\Delta_{g} u +\mathtt{C}(n,\beta)u=
\lambda u\left|\frac{\nabla_{g} d_1}{d_1}-\frac{\nabla_{g} d_2}{d_2}\right|^2+|u|^{p-2}u, & \hbox{ in } \mathbb{S}_+^{n} \\
u=0, & \hbox{ on } \partial\mathbb{S}_+^{n},
\end{array}
\right.\eqno{(\mathscr{P}_{\mathbb{S}_+^{n}})}
$$
where $g$ is the natural Riemannian structure  on the standard unit sphere $\mathbb S^n$ inherited by $\mathbb R^{n+1}$, 
 $p\in (2,2^*)$, $\lambda \in \left[0,\frac{(n-2)^2}{4}\right)$ is fixed and 
$\mathtt{C}(n,\beta)=(n-1)(n-2)\frac{7\pi^2-3\left(\beta+\frac{\pi}{2}\right)^2}{2\pi^2\left(\pi^2-\left(\beta+\frac{\pi}{2}\right)^2\right)}$; hereafter, $x_0=(0,...,0,1)$ is the north pole of $\mathbb S^n $  and $\beta=\max\{d_{g}(x_0,x_1), d_{g}(x_0,x_2)\}$.  


\begin{theorem}\label{theorem-application-gomb} Let ${\mathbb{S}}_+^{n}$ be the  open upper hemisphere $(n\geq 3)$,  $S=\{x_1,x_2\}\subset {\mathbb{S}}_+^{n}$ be the set of poles and $p\in (2,2^*)$. The following statements hold: 
\begin{itemize}
	\item[(i)] Problem $(\mathscr{P}_{\Sg})$ has infinitely many weak solutions in $H^1_{g}(\Sg)$.
	In addition, if  $x_1=(a,0,...,0,b)$ and $x_2=(-a,0,...,0,b)$ for some $a,b\in \mathbb R$ with $a^2+b^2=1$ and $b>0,$ then  problem $(\mathscr{P}_{\Sg})$ has a sequence $\{u_k\}_{k\in \mathbb N}$ of distinct weak solutions in $H^1_{g}(\Sg)$ of the form $$u_k:=u_k\left(y_1,\sqrt{y_2^2+...+y_n^2},y_{n+1}\right)=u_k\left(y_1,\sqrt{1-y_1^2-y_{n+1}^2},y_{n+1}\right).$$
	\item[(ii)] If $n=5$ or $n\geq 7,$ and $x_1=(a,0,...,0,b)$, $x_2=(-a,0,...,0,b)$ for some $a,b\in \mathbb R$ with $a^2+b^2=1$ and $b>0,$ then there exists at least $\ds s_n=\left[\frac{n}{2}\right]+(-1)^{n-1}-2$ sequences  of sign-changing weak solutions of $(\mathscr{P}_{\Sg})$ in $H^1_{g}(\Sg)$ whose elements mutually differ by their symmetries. 
\end{itemize}
	
\end{theorem}
 \textit{Proof.} Fix $\lambda \in \left[0,\frac{(n-2)^2}{4}\right)$ arbitrarily. The energy functional $\mathcal{E}:H^1_{g}(\Sg)\to \R$ associated with  problem $(\mathscr{P}_{\Sg})$ is given by  $$\mathcal{E}(u)=\frac{1}{2}\|u\|_{\mathtt{C}(n,\beta)}^2-\frac{\lambda}{2}\ig{u^2\left|\frac{\nabla_{g} d_1}{d_1}-\frac{\nabla_{g} d_2}{d_2}\right|^2}-\frac{1}{p}\ig{|u|^p}.$$ 
It is clear that $\mathcal{E}\in C^1(H^1_{g}(\Sg), \R)$ and its critical points are precisely the weak solutions of  $(\mathscr{P}_{\Sg})$.

(i)   We notice that the embedding $H_g^1(\mathbb S^n_+)\hookrightarrow L^p(\mathbb S^n_+)$ is compact for every $p\in (2,2^*)$, see e.g. Hebey \cite{Hebey-book}. By means of Corollary \ref{corollary-ezleszajoagombon}, one can prove that the  functional $\mathcal{E}$ satisfies the assumptions of the symmetric version of the mountain pass theorem,  see e.g. Jabri \cite[Theorem 11.5]{Jabri-konyv} or Rabinowitz \cite[Theorem 9.12]{Rabinowitz-konyv}, thus there exists a sequence of distinct critical points of $\mathcal E$ which are weak solutions of problem $(\mathscr{P}_{\Sg})$ in $H^1_{g}(\Sg)$. 

In particular, let $x_1=(a,0,...,0,b)$ and $x_2=(-a,0,...,0,b)$ for some $a,b\in \mathbb R$ with $a^2+b^2=1$ and $b>0.$ We notice that in this case $\beta=d_{g}(x_0,x_1)=d_{g}(x_0,x_2)=\arccos b.$ We shall prove that the energy functional $\mathcal{E}$ is invariant w.r.t. the group $G_0={\rm id}_\R\times  O(n-1)\times {\rm id}_\R$ via the action  $$\zeta u(x)=u(\zeta^{-1}x)$$ for every $u\in H^1_g(\Sg)$,  $\zeta \in G_0$ and $x\in \mathbb S_+^n$. 
First, since $\zeta \in G_0$ is an isometry on $\mathbb R^{n+1}$,   a change of variables easily implies that $$u\mapsto \frac{1}{2}\|u\|_{\mathtt{C}(n,\beta)}^2-\frac{1}{p}\ig{|u|^p}$$  
is $G_0$-invariant. Thus, it remains to focus on the $G_0$-invariance of the  functional $$u\mapsto \ig{u^2\left|\frac{\nabla_g d_1}{d_1}-\frac{\nabla_g d_2}{d_2}\right|^2}.$$
To do this, we  recall that $$\left|\frac{\nabla_g d_1}{d_1}-\frac{\nabla_g d_2}{d_2}\right|^2=\frac{1}{d_1^2}+\frac{1}{d_2^2}-2\frac{\langle \nabla_g d_1, \nabla_g d_2\rangle}{d_1d_2}.$$ 
and  $\nabla_g d_g(\cdot,y)(x)=-\frac{\exp_x^{-1}(y)}{d_g(x,y)}\text{ for every }x,y\in \Sg, \text{ }x\neq y$. According to Udri\c ste  \cite[p. 19]{Udriste-Book94}, one has  
%
$$ \exp_x^{-1}x_i=\frac{d_i(x_i-x\cos d_i)}{\sin d_i}, \ \ \ \ i\in\{1,2\}, \ x\in \Sg\setminus \{x_i\}.$$
Therefore, \begin{equation}\label{expfontos}\nabla_g d_i(x)=\nabla_g d_g(x,x_i)=-\frac{\exp_x^{-1}(x_i)}{d_i}=\frac{x\cos d_i-x_i}{\sin d_i}, \ \ \ \ i\in\{1,2\}, \ x\in \Sg\setminus \{x_i\}.
\end{equation}
Let $\zeta \in G_0$, $ i\in\{1,2\}$ and $ x\in \Sg\setminus \{x_i\}$ be fixed. Since $\zeta x_i=x_i$, it follows that   $$d_i(\zeta x)=d_g(\zeta x,x_i)=d_g(\zeta x,\zeta x_i)=d_g(x,x_i)=d_i(x),$$ 
and by (\ref{expfontos}), 
\begin{align*}
\langle \nabla_g d_g(\zeta x,x_1), \nabla_g d_g(\zeta x,x_2)\rangle=\langle \nabla_g d_g(x,x_1), \nabla_g d_g( x,x_2)\rangle.
\end{align*}
Summing up the above properties (combined with a trivial change of variable), it follows that the energy functional  $\mathcal E$ is $G_0$-invariant, i.e.,  $\mathcal E(\zeta u)=\mathcal E( u)$ for every $u\in H^1_g(\Sg)$ and $\zeta \in G_0$. 

We now can apply the same variational argument as above for the functional $\mathcal E_0=\mathcal E|_{H_{G_0}(\Sg)}$ where $H_{G_0}(\Sg)=\left\{u\in H^1_{g}(\Sg):{\zeta}u=u \hbox{ for every }{\zeta}\in G_0 \right\}.$ Accordingly, one can find a sequence $\{u_k\}_{k\in \mathbb N}\subset H_{G_0}(\Sg)$ of distinct critical points of $\mathcal E_0$. Moreover, due to the principle of symmetric criticality of Palais \cite{Palais}, the critical points of $\mathcal E_0$ are also critical points for the original energy functional $\mathcal E$, thus weak solutions of problem $(\mathscr{P}_{\Sg})$.  Since $u_k$ are $G_0$-invariant functions, they have the form $u_k:=u_k\left(y_1,\sqrt{y_2^2+...+y_n^2},y_{n+1}\right)=u_k\left(y_1,\sqrt{1-y_1^2-y_{n+1}^2},y_{n+1}\right),$ $k\in \mathbb N$. 

(ii)  Let $n=5$ or $n\geq 7$, and denote by $s_n=\left[\frac{n}{2}\right]+(-1)^{n-1}-2$. (Note that $s_6=0$.) For every $j\in \{1,...,s_n\}$ we define $$G^n_j=\left\{
\begin{array}{ll}
O(j+1)\times O(n-2j-3)\times O(j+1), & \hbox{  if } j\neq \frac{n-3}{2}; \\
O\left(\frac{n-1}{2}\right)\times O\left(\frac{n-1}{2}\right), & \hbox{ if } j=\frac{n-3}{2},
\end{array}
\right.$$
where $O(k)$ is the orthogonal group in $\mathbb R^k.$ 
For a fixed $G^n_j,$ we define the function $\tau_j$ associated to $G^n_j$ as $$\tau_j(\sigma)=\left\{
\begin{array}{ll}
(\sigma_3,\sigma_2,\sigma_1), &\hbox{ if }j\neq \frac{n-3}{2} \hbox{ and } \sigma=(\sigma_1,\sigma_2,\sigma_3)\hbox{ with }\sigma_1,\sigma_2\in\R^{j+1},\sigma_2\in \R^{n-2j-3}; \\
(\sigma_3,\sigma_1), & \hbox{ if }j= \frac{n-3}{2} \hbox{ and }\sigma=(\sigma_1,\sigma_3) \hbox{ with }\sigma_1,\sigma_3 \in \R^{\frac{n-1}{2}}.
\end{array}
\right.$$ 
Note that $\tau_j\notin G^n_j$, $\tau_j G^n_j \tau_j^{-1}=G^n_j$ and $\tau_j^2={\rm id}_{\R^{n-1}}$. Similarly as in Krist\'aly  \cite{Kristaly-DCDS2009}, we introduce the action of the group $$G^n_{j,\tau_j}={\rm id}_\R\times \langle G^n_j,\tau_j\rangle\times {\rm id}_\R \subset O(n+1)$$ on the space $H^1_g(\Sg)$ by
\begin{equation}\label{csoporthatas}\zeta u(x)=u(\zeta^{-1}x), \ \ \ \ (\widetilde{\tau}_j\zeta)u(x)=-u(\zeta^{-1}\widetilde{\tau}_j^{-1} x),
\end{equation}
for every $\zeta \in \widetilde G^n_j={\rm id}_\R\times  G^n_j\times {\rm id}_\R$, $\widetilde{\tau}_j={\rm id}_\R\times \tau_j\times {\rm id}_\R$, $u\in H^1_{g}(\Sg)$ and $x\in \Sg$. We define the subspace of $H^1_g(\Sg)$ containing all the symmetric points w.r.t. the compact group $G^n_{j,\tau_j}$, i.e.,
$$H_{\G}(\Sg)=\left\{u\in H^1_{g}(\Sg):\tilde{\zeta}u=u \hbox{ for every }\tilde{\zeta}\in \G \right\}.$$
Note that (see Krist\'aly \cite[Theorem 3.1]{Kristaly-DCDS2009}) for every $j\neq k\in \{1,2,...,s_n\}$ one has 
\begin{equation}\label{metszet-nulla}
\ds H_{\G}(\Sg)\cap H_{G^n_{k,\tau_k}}(\Sg)=\{0\}.
\end{equation}

In a similar way as above, we can prove that the energy functional $\mathcal{E}$ is $\G$-invariant for every $j\in \{1,...,s_n\}$ (note that $\mathcal E$ is an even functional), where the group action on $H^1_{g}(\Sg)$ is given  by (\ref{csoporthatas}). Therefore, for every $j\in \{1,...,s_n\}$ there exists a sequence $\{u_k^j\}_{k\in \mathbb N}\subset H_{\G}(\Sg)$ of distinct critical points of $\mathcal E_j=\mathcal E|_{H_{\G}(\Sg)}$. Again by Palais \cite{Palais}, $\{u_k^j\}_{k\in \mathbb N}\subset H_{\G}(\Sg)$ are distinct critical points also for $\mathcal E$,  thus weak solutions for problem $(\mathscr{P}_{\Sg})$. It is clear that every  $u_k^j$ is sign-changing (see (\ref{csoporthatas})) and according to (\ref{metszet-nulla}), elements in different sequences have mutually different symmetry properties. 
\hfill $\Box$
\begin{remark}\rm 
 For $n=6$ in Theorem 
\ref{theorem-application-gomb} (ii), one has $s_6=0$; therefore, in this case we cannot apply the above group-theoretical argument to guarantee the existence of sign-changing solutions for problem $(\mathscr{P}_{\Sg})$. 
\end{remark}

\section{Concluding remarks}\label{section-final}
 In the present paper we presented some multipolar Hardy inequalities on complete Riemannian manifolds by exploring the presence of the curvature and giving some applications in the theory of  elliptic equations involving bipolar potentials; as far as we know, this is the first study in such a geometrical setting. During the preparation of the manuscript we faced several problems which, - in our opinion, - are worth to be tackled in forthcoming investigations. In the sequel, we shall formulate some of them:

\begin{itemize}
	\item[(a)]  As we already pointed out in Remark \ref{remark-elso} (b), the optimality of $\frac{(n-2)^2}{m^2}$ in (\ref{originial-general}) for generic Riemannian manifolds is not yet understood for $m\geq 3$ which requires further studies. We notice that multipolar inequalities involving non-uniform weights on complete Riemannian manifolds can also be obtained, following Devyver, Fraas and Pinchover \cite{DFP}. 
	
	
	\item[(b)] For simplicity reasons, in \S \ref{section-appl} we considered only some model elliptic problems with familiar growth assumptions, i.e., sublinear and subcritical pure power term. However,   multipolar Hardy inequalities (cf. Theorems \ref{main-main1} and \ref{masodik}) allow to study other classes of elliptic problems involving other type of nonlinear terms (critical,  concave-convex, etc.). 
	\item[(c)] A challenging problem is to study the heat equation involving multiple poles on strip-like domains or curved tubes (embedded into appropriate Riemannian manifolds). We notice that  in the Euclidean setting such equations have been investigated by Baras and Goldstein \cite{Baras-Golstein},  Krej{\v{c}}i{\v{r}}{\'{\i}}k  and Zuazua
	\cite{Kr-Zuazua-JMPA, Kr-Zuazua-JDE} via Hardy-type inequalities; see also references therein. We notice that deep studies already exist concerning linear heat equations on  Riemannian manifolds having non-negative Ricci curvature which is related to the Perelman's volume non-collapsing result, see Ni \cite{Ni-JGA}. 
	
\end{itemize}

 \vspace{0.8cm}
 
 \noindent \textbf{Acknowledgment.} The research of Cs. Farkas and A. Krist\'aly is supported by a grant of
 the Romanian National Authority for Scientific Research,
 CNCS-UEFISCDI, project no. PN-II-ID-PCE-2011-3-0241. Both authors are grateful to the Dipartimento di Matematica e Informatica, Universit\`a di Catania for the warm hospitality where this work has been initiated. Cs. Farkas also thanks the financial support of Gruppo Nazionale per l'Analisi Matematica, la Probabilit\`a e le loro Applicazioni (GNAMPA) of the Istituto Nazionale di Alta Matematica (INdAM). The authors thank Professor Enrique Zuazua and the three anonymous referees for their valuable remarks which improved the quality of our manuscript.

\bibliographystyle{amsalpha}
\bibliography{biblio}

\providecommand{\bysame}{\leavevmode\hbox to3em{\hrulefill}\thinspace}
\providecommand{\MR}{\relax\ifhmode\unskip\space\fi MR }
\providecommand{\MRhref}[2]{%
  \href{http://www.ams.org/mathscinet-getitem?mr=#1}{#2}
}
\providecommand{\href}[2]{#2}
\begin{thebibliography}{CnRnS11}

\bibitem[Adi13]{Adimurthi-2013-CCM}
Adimurthi, \emph{Best constants and {P}ohozaev identity for
  {H}ardy-{S}obolev-type operators}, Commun. Contemp. Math. \textbf{15} (2013),
  no.~3, 1250050, 23. \MR{3063553}

\bibitem[BDE08]{Bosi-Dolbeaut-Esteban-CPAA-2008}
R.~Bosi, J.~Dolbeault, and M.~J. Esteban, \emph{Estimates for the optimal
  constants in multipolar {H}ardy inequalities for {S}chr\"odinger and {D}irac
  operators}, Commun. Pure Appl. Anal. \textbf{7} (2008), no.~3, 533--562.
  \MR{2379440}

\bibitem[BG84]{Baras-Golstein}
P.~Baras and J.~A. Goldstein, \emph{The heat equation with a singular
  potential}, Trans. Amer. Math. Soc. \textbf{284} (1984), no.~1, 121--139.
  \MR{742415}

\bibitem[BH99]{Bridson-Haefliger}
M.~R. Bridson and A.~Haefliger, \emph{Metric spaces of non-positive curvature},
  Grundlehren der Mathematischen Wissenschaften [Fundamental Principles of
  Mathematical Sciences], vol. 319, Springer-Verlag, Berlin, 1999. \MR{1744486}

\bibitem[BPW01]{Bartsch-Pankov-Wang}
T.~Bartsch, A.~Pankov, and Z.-Q. Wang, \emph{Nonlinear {S}chr\"odinger
  equations with steep potential well}, Commun. Contemp. Math. \textbf{3}
  (2001), no.~4, 549--569. \MR{1869104}

\bibitem[BW95]{Bartsch-Wang}
T.~Bartsch and Z.~Q. Wang, \emph{Existence and multiplicity results for some
  superlinear elliptic problems on {${\bf R}^N$}}, Comm. Partial Differential
  Equations \textbf{20} (1995), no.~9-10, 1725--1741. \MR{1349229}

\bibitem[Car97]{CarronHardy}
G.~Carron, \emph{In\'egalit\'es de {H}ardy sur les vari\'et\'es riemanniennes
  non-compactes}, J. Math. Pures Appl. (9) \textbf{76} (1997), no.~10,
  883--891. \MR{1489943}

\bibitem[CH06]{Cao-Han}
D.~Cao and P.~Han, \emph{Solutions to critical elliptic equations with
  multi-singular inverse square potentials}, J. Differential Equations
  \textbf{224} (2006), no.~2, 332--372. \MR{2223721}

\bibitem[CnRnS11]{spanyolok}
J.~F. Cari\~nena, M.~F. Ra\~nada, and M.~Santander, \emph{The quantum free
  particle on spherical and hyperbolic spaces: a curvature dependent approach},
  J. Math. Phys. \textbf{52} (2011), no.~7, 072104, p. 21.

\bibitem[CZ13]{cazacuzuazua}
C.~Cazacu and E.~Zuazua, \emph{Improved multipolar {H}ardy inequalities},
  Studies in phase space analysis with applications to {PDE}s, Progr. Nonlinear
  Differential Equations Appl., vol.~84, Birkh\"auser/Springer, New York, 2013,
  pp.~35--52. \MR{3185889}

\bibitem[dC92]{doCarmo}
M.~P. do~Carmo, \emph{Riemannian geometry}, Mathematics: Theory \&
  Applications, Birkh\"auser Boston, Inc., Boston, MA, 1992, Translated from
  the second Portuguese edition by Francis Flaherty. \MR{1138207}

\bibitem[DD14]{dambrosiohardyRiemann}
L.~D'Ambrosio and S.~Dipierro, \emph{Hardy inequalities on {R}iemannian
  manifolds and applications}, Ann. Inst. H. Poincar\'e Anal. Non Lin\'eaire
  \textbf{31} (2014), no.~3, 449--475. \MR{3208450}

\bibitem[Dev14]{D}
B.~Devyver, \emph{A spectral result for {H}ardy inequalities}, J. Math. Pures
  Appl. (9) \textbf{102} (2014), no.~5, 813--853. \MR{3271291}

\bibitem[DFP14]{DFP}
B.~Devyver, M.~Fraas, and Y.~Pinchover, \emph{Optimal {H}ardy weight for
  second-order elliptic operator: an answer to a problem of {A}gmon}, J. Funct.
  Anal. \textbf{266} (2014), no.~7, 4422--4489. \MR{3170212}

\bibitem[FKV15]{FKVCalculus}
C.~Farkas, A.~Krist{\'a}ly, and C.~Varga, \emph{Singular {P}oisson equations on
  {F}insler-{H}adamard manifolds}, Calc. Var. Partial Differential Equations
  \textbf{54} (2015), no.~2, 1219--1241. \MR{3396410}

\bibitem[FMT07]{Felli-Marchini-Terracini-2007JFA}
V.~Felli, E.~M. Marchini, and S.~Terracini, \emph{On {S}chr\"odinger operators
  with multipolar inverse-square potentials}, J. Funct. Anal. \textbf{250}
  (2007), no.~2, 265--316. \MR{2352482}

\bibitem[GHL87]{GallotHulinLafontain}
S.~Gallot, D.~Hulin, and J.~Lafontaine, \emph{Riemannian geometry},
  Universitext, Springer-Verlag, Berlin, 1987. \MR{909697}

\bibitem[GHN12]{Kinaiak}
Q.~Guo, J.~Han, and P.~Niu, \emph{Existence and multiplicity of solutions for
  critical elliptic equations with multi-polar potentials in symmetric
  domains}, Nonlinear Anal. \textbf{75} (2012), no.~15, 5765--5786.
  \MR{2948296}

\bibitem[Heb99]{Hebey-book}
E.~Hebey, \emph{Nonlinear analysis on manifolds: {S}obolev spaces and
  inequalities}, Courant Lecture Notes in Mathematics, vol.~5, New York
  University, Courant Institute of Mathematical Sciences, New York; American
  Mathematical Society, Providence, RI, 1999. \MR{1688256}

\bibitem[Jab03]{Jabri-konyv}
Y.~Jabri, \emph{The mountain pass theorem}, Encyclopedia of Mathematics and its
  Applications, vol.~95, Cambridge University Press, Cambridge, 2003, Variants,
  generalizations and some applications. \MR{2012778}

\bibitem[KKOR10]{oroszok}
V.~V. Kudryashov, Yu.~A. Kurochkin, E.~M. Ovsiyuk, and V.~M. Red'kov,
  \emph{Classical particle in presence of magnetic field, hyperbolic
  {L}obachevsky and spherical {R}iemann models}, SIGMA Symmetry Integrability
  Geom. Methods Appl. \textbf{6} (2010), Paper 004, p. 34.

\bibitem[Kli95]{Klingenberg}
W.~P.~A. Klingenberg, \emph{Riemannian geometry}, second ed., de Gruyter
  Studies in Mathematics, vol.~1, Walter de Gruyter \& Co., Berlin, 1995.
  \MR{1330918}

\bibitem[K{\"O}09]{Kombe-Ozaydin-TAMS2019}
I.~Kombe and M.~{\"O}zaydin, \emph{Improved {H}ardy and {R}ellich inequalities
  on {R}iemannian manifolds}, Trans. Amer. Math. Soc. \textbf{361} (2009),
  no.~12, 6191--6203. \MR{2538592}

\bibitem[K{\"O}13]{Kombe-Ozaydin-TAMS2013}
\bysame, \emph{Hardy-{P}oincar\'e, {R}ellich and uncertainty principle
  inequalities on {R}iemannian manifolds}, Trans. Amer. Math. Soc. \textbf{365}
  (2013), no.~10, 5035--5050. \MR{3074365}

\bibitem[KR16]{Kristaly-Repovs-2016-CCM}
A.~Krist\'aly and D.~Repovs, \emph{Quantitative {R}ellich inequalities on
  {F}insler-{H}adamard manifolds}, Commun. Contemp. Math. \textbf{18} (2016),
  no.~6, 1650020, 17. \MR{3547105}

\bibitem[Kri09]{Kristaly-DCDS2009}
A.~Krist{\'a}ly, \emph{Asymptotically critical problems on higher-dimensional
  spheres}, Discrete Contin. Dyn. Syst. \textbf{23} (2009), no.~3, 919--935.
  \MR{2461832}

\bibitem[KZ10]{Kr-Zuazua-JMPA}
D.~Krej{\v{c}}i{\v{r}}{\'{\i}}k and E.~Zuazua, \emph{The {H}ardy inequality and
  the heat equation in twisted tubes}, J. Math. Pures Appl. (9) \textbf{94}
  (2010), no.~3, 277--303. \MR{2679028}

\bibitem[KZ11]{Kr-Zuazua-JDE}
\bysame, \emph{The asymptotic behaviour of the heat equation in a twisted
  {D}irichlet-{N}eumann waveguide}, J. Differential Equations \textbf{250}
  (2011), no.~5, 2334--2346. \MR{2756066}

\bibitem[Lie05]{Lieb}
E.~H. Lieb, \emph{The stability of matter: from atoms to stars}, fourth ed.,
  Springer, Berlin, 2005, Selecta of Elliott H. Lieb, Edited by W. Thirring,
  and with a preface by F. Dyson.

\bibitem[Ni04]{Ni-JGA}
L.~Ni, \emph{The entropy formula for linear heat equation}, J. Geom. Anal.
  \textbf{14} (2004), no.~1, 87--100. \MR{2030576}

\bibitem[Pal79]{Palais}
R.~S. Palais, \emph{The principle of symmetric criticality}, Comm. Math. Phys.
  \textbf{69} (1979), no.~1, 19--30. \MR{547524}

\bibitem[PRS08]{Pigola-Rigoli-Setti2008}
S.~Pigola, M.~Rigoli, and A.~G. Setti, \emph{Vanishing and finiteness results
  in geometric analysis}, Progress in Mathematics, vol. 266, Birkh\"auser
  Verlag, Basel, 2008, A generalization of the Bochner technique. \MR{2401291}

\bibitem[Rab86]{Rabinowitz-konyv}
P.~H. Rabinowitz, \emph{Minimax methods in critical point theory with
  applications to differential equations}, CBMS Regional Conference Series in
  Mathematics, vol.~65, Published for the Conference Board of the Mathematical
  Sciences, Washington, DC; by the American Mathematical Society, Providence,
  RI, 1986. \MR{845785}

\bibitem[Rab92]{Rabinowitz-ZAMP}
\bysame, \emph{On a class of nonlinear {S}chr\"odinger equations}, Z. Angew.
  Math. Phys. \textbf{43} (1992), no.~2, 270--291. \MR{1162728}

\bibitem[She97]{Shen-volume}
Z.~Shen, \emph{Volume comparison and its applications in {R}iemann-{F}insler
  geometry}, Adv. Math. \textbf{128} (1997), no.~2, 306--328. \MR{1454401}

\bibitem[Udr94]{Udriste-Book94}
C.~Udri{\c{s}}te, \emph{Convex functions and optimization methods on
  {R}iemannian manifolds}, Mathematics and its Applications, vol. 297, Kluwer
  Academic Publishers Group, Dordrecht, 1994. \MR{1326607}

\bibitem[Wil96]{Willem-book}
M.~Willem, \emph{Minimax theorems}, Progress in Nonlinear Differential
  Equations and their Applications, 24, Birkh\"auser Boston, Inc., Boston, MA,
  1996. \MR{1400007}

\bibitem[WX07]{Wu-Xin}
B.~Y. Wu and Y.~L. Xin, \emph{Comparison theorems in {F}insler geometry and
  their applications}, Math. Ann. \textbf{337} (2007), no.~1, 177--196.
  \MR{2262781}

\bibitem[Xia14]{XIA-JMAA_2014}
C.~Xia, \emph{Hardy and {R}ellich type inequalities on complete manifolds}, J.
  Math. Anal. Appl. \textbf{409} (2014), no.~1, 84--90. \MR{3095019}

\bibitem[YSK14]{Yang-Su-Kong-CCM2014}
Q.~Yang, D.~Su, and Y.~Kong, \emph{Hardy inequalities on {R}iemannian manifolds
  with negative curvature}, Commun. Contemp. Math. \textbf{16} (2014), no.~2,
  1350043, 24. \MR{3195155}

\end{thebibliography}
\end{document}